\documentclass[10pt,a4paper]{article}
\usepackage{amsmath, amscd, amssymb, amsthm, amscd}
\newtheorem{thm}{Theorem}[section]
\newtheorem{lemma}{Lemma}[section]
\newtheorem{prop}{Proposition}[section]
\newtheorem{defn}{Definition}[section]

\begin{document}
\title{Reciprocity Laws for the Higher Tame Symbol and the Witt Symbol on an Algebraic Surface}\author{Kirsty Syder}\date{}
\maketitle
\begin{abstract}
Parshin's higher Witt pairing on an arithmetic surface can be combined with the higher tame pairing to form a symbol taking values in the absolute abelian Galois group of the function field. We prove reciprocity laws for this symbol using techniques of Morrow for the Witt symbol and Romo for the higher tame symbol.
\end{abstract}

\section{Introduction}
The study of higher local fields was initiated in the 1970s by Y. Ihara, with further work done by A. N. Parshin in the positive characteristic case and K. Kato in the general case. We recall the inductive definition: an $n$-dimensional local field $K$ is a complete discrete valuation field such that its residue field is an $(n-1)$-dimensional local field, where a one-dimensional local field is a local field and a zero-dimensional local field is a finite field. As in the one-dimensional case, all $n$-dimensional local fields can be classified as a finite extension of one of the following types of field:
\begin{enumerate}
\item{$\mathbb{F}_{q}((t_{1}))\dots((t_{n}))$, where $q$ is a prime power;}
\item{$K((t_{1}))\dots((t_{n-1}))$, where $K$ is a finite extension of $\mathbb{Q}_{p}$;}
\item{$K\{\{t_{1}\}\}\dots\{\{t_{i}\}\}((t_{i+1}))\dots((t_{n-1}))$, a mixed characteristic local field, where $K$ is a finite extension of $\mathbb{Q}_{p}$.}\end{enumerate}

See \cite{IHLF}, section one for full definitions and proofs. This paper will study two-dimensional fields of positive characteristic, i.e. fields isomorphic to $\mathbb{F}_{q}((t_{1}))((t_{2}))$ for some prime power $q$.\\
The method of using the Witt pairing to define a reciprocity map for wildly ramified extensions of fields of positive characteristic was first developed by Kawada and Satake in their paper \cite{KS}. They developed the class field theory for local fields and function fields of positive characteristic, and Parshin's method discussed below is a higher-dimensional generalisation of Kawada and Satake's method.\\
In his papers \cite{P4} and \cite{P5}, Parshin developed a reciprocity map for higher local fields by gluing together three separate maps for unramified, tamely ramified and wildly ramified extensions. See \cite{IHLF} section seven for a review of this theory, and also \cite{F7}. The map for tamely ramified extensions came from the higher tame symbol, which is a higher dimensional generalisation of the tame symbol, 
\[\{f,g\} = (-1)^{v(f)v(g)}\frac{f^{v(g)}}{g^{v(f)}}\]
for $f$, $g$ elements of a local field with valuation $v$. This symbol has been studied extensively, and the reciprocity laws described below proved using several different methods.\\
The map for wildly ramified extensions is the Artin-Schreier-Witt pairing, see section four for a full definition.\\
 We prove reciprocity laws for this symbol relying mainly on the recent work of Morrow in \cite{Mor3} and the work of Kawada and Satake in \cite{KS}. Morrow proves reciprocity laws for an arithmetic surface, i.e. the characteristic zero case as compared to our positive characteristic case, and so the difficulty is mostly in dealing with the mixed characteristic fields which occur in the case of an arithmetic surface. \\
\bigskip

We will now describe the global situation. As in the classical one-dimensional case we study a higher global field using its localisations. Let $X$ be an algebraic surface over a finite field, that is a connected smooth projective surface over $\mathbb{F}_{q}$ for some prime power $q$. Let $F$ be the function field of $X$. For a point $x \in X$ and a curve $y \subset X$, we can associate certain fields. To a pair $(x,y)$ where $x \in y \subset X$, we have a product of two-dimensional local fields $F_{x,y}$ - see the next section for more details. \\
We also have the semi-global fields $F_{y} = k(y)((t_{y}))$ where $k(y)$ is the function field of the curve $y$ and $t_{y}$ is a local parameter for $y$ in $F_{x,y}$, and $F_{x}$, the smallest ring containing the completed local ring $\hat{\mathcal{O}}_{X,x}$ and $F$. The class field theory of these fields as well as that of $F$ is of interest - see e.g. the works of Kato and Saito, \cite{KSCFT1}, \cite{KSCFT2} and \cite{KSCFT3}.\\
\bigskip

The structure of this paper is as follows. The first section will discuss the global situation and fields related to algebraic surfaces in more detail, then define the Milnor $K$-groups and topological $K$-groups of a field. The reciprocity map acts from these groups to the absolute abelian Galois group, so these are of key importance. The next two sections define the higher tame symbol and the Witt symbol and discuss their basic properties. Section five then recalls the relation of the symbols to Galois theory for a higher local field, and shows how this can be used to get a map to the absolute abelian Galois group of a semi-global field. The final two sections use this to glue together the symbols in this Galois group, and prove reciprocity laws around a point and along a curve on the algebraic surface $X$. The proofs of the reciprocity laws for the Witt symbol are new, and use a method similar to Morrow's work on reciprocity laws for differentials, adapted for our different situation which also uses methods of Witt, Kawada and Satake. The gluing together of the symbols in the global situation is also new work, we use Galois theory and appeal to the one-dimensional case.\\
When gluing the two symbols together to prove the reciprocity laws, we first prove separate reciprocity laws for each symbol, then use the homomorphism to the absolute abelian Galois group to get a reciprocity law for the glued symbols. This method is used because the higher tame symbol takes values in the multiplicative group $\mathbb{F}_{q}^{\times}$, so the reciprocity law is in the form of a product, and the Witt symbol takes values in the additive group $\mathbb{F}_{q}$ with reciprocity law in the form of a sum. Hence it is more clear to keep the two symbols separate outside of the Galois group and combine the reciprocity laws after applying the reciprocity map.\\
Finally, note the relation to other versions of class field theory for schemes. The methods used in this paper are based on the work of Kawada and Satake and then Parshin's generalisation of their local methods to the case of higher local fields. We develop this theory further by proving reciprocity laws in semi-global cases - i.e. along curves and around points - as a first step to using this method to prove duality theorems and get a full version of class field theory as in \cite{KS}.
We provide a first proof of such reciprocity laws for the Witt symbol, and an explicit proof for the higher tame symbol - both take into account singular curves on the algebraic surface, a case which has not always been looked at in past study of the higher tame symbol.
 This will be a simpler, more explicit class field theory than that developed by Kato and Saito in \cite{KSCFT1}. Another more explicit version is that developed first by Wiesend and completed by Kerz and Schmidt - see \cite{KS3} and \cite{MK}. Similarly to Parshin's approach, these papers get the unramified part of the reciprocity map from the Frobenius element at each closed point, but then  the ramified part comes from the theory of curves contained in the scheme (i.e. one-dimensional subschemes). This approach works for any dimension of scheme, so is very explicit in defining a simple class group and reciprocity map relying only on closed points and the one-dimensional theory associated to curves on the surface. Another version of class field theory for two-dimensional schemes is mentioned by Fesenko in section 34 of \cite{F3}, where he applies his generalisation from \cite{F7} and \cite{F8} of Neukirch's explicit class field theory - \cite{N} - to higher local fields.\\
\\
\textbf{Acknowledgements}\\
\\
I am grateful to Alberto C\'amara, Matthew Morrow and Thomas Oliver for many useful conversations and guidance during my research. I am also very grateful to my supervisor Ivan Fesenko for suggesting the area of work and providing many comments and improvements. I am supported by an EPSRC grant at the University of Nottingham.

\section{Algebraic Surfaces and Milnor $K$-groups}

Let $X$ be an algebraic surface over a finite field $k = \mathbb{F}_{q}$. If $x \in X$ is a smooth closed point on a reduced irreducible curve $y \subset X$, we can associate a two-dimensional local field to the pair $(x,y)$. Let $t_{y} \in \mathcal{O}_{X,x}$ be a local parameter for $y$ and $u_{x,y} \in \mathcal{O}_{y,x}$ a local parameter for $x$, where $\mathcal{O}_{y}$ is the local ring of $X$ at $y$ and $\mathcal{O}_{y,x}$ its completion at the point $x \in y$. Then the field
\[F_{x,y} = k(x)((u_{x,y}))((t_{y}))\]
is a two-dimensional local field over $k(x)$. $F_{x,y}$ can be constructed through a series of completions and localisations independent of the chosen local parameters, see \cite{CHDLF} section 3 for full details of this process.\\
If $x \in y$ is not a smooth point, let $z$ be a branch of $y$ at $x$ and define $F_{x,z}$ as above. Then we let $y(x)$ be the set of branches of $y$ at $x$ and define \[F_{x,y} = \prod_{z \in y(x)}F_{x,z}.\]
The Milnor $K$-groups of a higher local field play an important role in class field theory, taking the place of the multiplicative group in the one dimensional case to obtain the reciprocity map. As will be clear from the definition, any symbol on a field satisfying the Steinberg property will reduce to a symbol on the Milnor $K$-group of the field.
\begin{defn}
For a field $F$, let
\[I_{n} := \{\alpha_{1} \otimes \dots \otimes \alpha_{n} \in (F^{\times})^{\otimes n} : \alpha_{i} + \alpha_{j} = 1, \text{ for some } 1 \leq i,j \leq n \}.\]
Define the $n^{th}$ Milnor $K$-group of $F$ as
\[K_{n}(F) = (F^{\times})^{\otimes n}/I_{n}.\]\end{defn}
We denote elements of $K_{n}(F)$ by $\{\alpha_{1}, \dots, \alpha_{n}\}$ and write the group law additively. For a product of fields $F_{x,y}$ at a singular point $x$, the group $K_{n}(F_{x,y})$ will be a product of the $K$-groups $K_{n}(F_{x,z})$ at the branches $z \in y(x)$.\\
We also define a relative $K$-group for a complete discrete valuation field $F$, $K_{n}(\mathcal{O}_{F}, \mathfrak{p}_{F}) : = $ ker$(K_{n}(\mathcal{O}_{F}) \to K_{n}(\mathcal{O}_{F}/\mathfrak{p}_{F}))$, the kernel of the reduction map.\\
\\
For full details on the following definition, see \cite{F6}. Endow $K_{n}(F_{x,z})$ with the strongest topology such that negation and the symbol map $(F_{x,z}^{\times})^{n} \to K_{n}(F_{x,z})$ are sequentially continuous. The topology on the multiplicative group of a two dimensional local field of positive characteristic is defined as the product of the discrete topologies on the groups generated by the local parameters $u_{x,z}$, $t_{z}$, the discrete topology on $k_{z}(x)^{\times} = (\mathcal{O}_{x,z}/\mathfrak{p}_{x,z})^{\times}$ and the topology on the principal unit group induced from the topology on $F_{x,z}$. The topology on the additive group $F_{x,z}$ is defined by a system of open neighbourhoods of zero lifted from such a system in the residue field.
\begin{defn}
Define the $n^{th}$ topological Milnor K-group, $K_{n}^{top}(F_{x,z})$, as the quotient of $K_{n}(F_{x,z})$ by the intersection of all its neighbourhoods of zero. \end{defn}
Fesenko proves that $K_{n}^{top}(F_{x,z}) = K_{n}(F_{x,z})/\cap_{l\geq 1}lK_{n}(F_{x,z})$ in \cite{F6}. See \cite[6]{IHLF} for details of how symbol maps are used to study the structure of topological Milnor $K$-groups.

\section{The Higher Tame Symbol on an Algebraic Surface}
As before, let $X$ be an algebraic surface over $k$ and $x \in y \subset X$ a point on a curve contained in $X$. The higher tame symbol takes values in $k_{z}(x)$. First let $x$ be a smooth point of $y$. If $f$, $g$ and $h$ are elements of $F_{x,y}$, then the higher tame symbol is expressed as
\[ (f,g,h)_{x,y} = (-1)^{\alpha_{x,y}}\left( \frac{f^{v_{y}(g)\bar{v}_{x}(h) - v_{y}(h)\bar{v}_{x}(g)}}{g^{v_{y}(f)\bar{v}_{x}(h) - v_{y}(h)\bar{v}_{x}(f)}}h^{v_{y}(f)\bar{v}_{x}(g) - v_{y}(g)\bar{v}_{x}(f)} \right) \text{ mod } \mathfrak{p}_{x,y}\]
where:\\
\[\alpha_{x,y} = v_{y}(f)v_{y}(g)\bar{v}_{x}(h)+v_{y}(f)v_{y}(h)\bar{v}_{x}(g) + v_{y}(g)v_{y}(h)\bar{v}_{x}(f) +\]\
 \[v_{y}(f)\bar{v}_{x}(g)\bar{v}_{x}(h) + v_{y}(g)\bar{v}_{x}(f)\bar{v}_{x}(h) + v_{y}(h)\bar{v}_{x}(f)\bar{v}_{x}(g);\]
$v_{y}$ is the surjective discrete valuation induced by $y$ and $\bar{v}_{x}$ is the function
\[\bar{v}_{x}: F_{x,y} \to \mathbb{Z}\]
defined by $\bar{v}_{x}(\beta) = v_{x,y}(p(t_{y}^{-v_{y}(\beta)}\beta))$, where $p$ is the projection map from $\mathcal{O}_{x,y}$ to $\bar{F}_{x,y}$ and $v_{x,y}$ is the discrete valuation on the local field $\bar{F}_{x,y}$. Finally, $\mathfrak{p}_{x,y}$ is the maximal ideal of $\mathcal{O}_{x,y}$.\\

Parshin introduced this symbol without the sign $(-1)^{\alpha_{x,y}}$ - this was first defined by Fesenko and Vostokov in their paper \cite{FV2}. They gave a simpler definition of the symbol using a two-dimensional discrete valuation. Let $\textbf{v}:= (\bar{v}_{x}, v_{y})= (v_{1}, v_{2})$. Then the symbol $(f_{1},f_{2},f_{3})_{x,y}$ is equal to the $(q-1)^{th}$ root of unity in $\mathbb{F}_{q}^{\times}$ which is equal to the residue of 
\[f_{1}^{b_{1}}f_{2}^{b_{2}}f_{3}^{b_{3}}(-1)^{b}\]
in $\mathbb{F}_{q}$, where \[b= \sum_{s,i<j}v_{s}(b_{i})v_{s}(b_{j})b_{i,j}^{s},\] 
$b_{j}$ is $(-1)^{j-1}$ multiplied by the determinant of the matrix $(v_{i}(f_{j}))$ with the $j^{th}$ column removed and $b_{i,j}^{s}$ is the determinant of the matrix with the $i^{th}$ and $j^{th}$ columns and $s^{th}$ row removed.\\

Notice the relation to the boundary homomorphism of $K$-theory - for $F$ an $n$-dimensional local field with first residue field $\bar {F}$, there is a map
\[ \delta : K_{i}(F) \to K_{i-1}(\bar{F}).\]
See \cite{FV}, chapter seven for details of this homomorphism. For $i = 3$, the boundary map coincides with the higher tame symbol up to a sign.\\

If $x$ is not a smooth point of the curve $y$, we can define the higher tame symbol for each local branch $z \in y(x)$ and then let $( \ ,\ ,\ )_{x,y} = \prod_{z \in y(x)} N_{k_{z}(x)/\mathbb{F}_{q}}( \ ,\ , \ )_{x,z}$.
\section{The Witt Symbol on an Algebraic Surface}
It can be seen that if we view the higher tame symbol as a pairing on $K_{2}(F)$ and $F^{\times}$, then $K_{2}(\mathcal{O}_{F}, \mathfrak{p}_{F})$ is the kernel of the left hand side: see \cite[Ch. 9, 2.2]{FV}, which proves that for one application of the boundary homomorphism the kernel is given by
\[K_{n}(\mathcal{O}_{F}, \mathfrak{p}_{F}) + U_{n}(F)\]
where $U_{n}(F)$ is the subgroup of $K_{n}(F)$ generated by symbols with all values in the unit group. We induct on this proposition: one application of the boundary symbol gives us the kernel $K_{3}(\mathcal{O}_{F},\mathfrak{p}_{F})+U_{3}(F)$, and similarly for the second application.
Thus the degeneracy on the left hand side comes from elements including a principal unit of $\mathcal{O}_{F}$, i.e. the group $K_{2}(\mathcal{O}_{F},\mathfrak{p}_{F})$ as required. 
Notice that when we write $K_{2}(\mathcal{O}_{F})$, the symbols will only have entries in the \emph{multiplicative} group in $\mathcal{O}_{F}$, a point which becomes more important when not working with fields. \\
As suggested by the name, the reciprocity map of higher class field theory associated to the tame symbol does not contain any information about wildly ramified extensions. For a field of positive characteristic, 
it is natural to use Witt vectors to study abelian $p$-extensions in view of Witt duality. Kawada, Satake and Parshin used this method for local and higher local fields of characteristic $p$, where the method gives much explicit information about the relation between the field and the extensions. This results in a theory far simpler than the earlier theory of Kato, see \cite{KK1}, \cite{KK2} and \cite{KK3}. We will use the method of Kawada, Satake and Parshin to define a wildly ramified reciprocity map to accompany the above higher tame symbol, then prove reciprocity laws for the symbols.\\

We first recall the definition of the continuous differential forms. For a pair $x \in y$, let $\mathfrak{m}_{x,y}$ be the maximal ideal of $\mathcal{O}_{x,y}$, generated by $t_{y}$ and $u_{x,y}$. Let $\phi: \hat{\mathcal{O}}_{X,x} \to \bar{F}_{x,y}$ be the quotient map for the ideal $t_{y}\hat{\mathcal{O}}_{X,x}$. Define the subgroups $P_{i}$ and $T_{j}$ in $\omega_{F_{x,y}/\mathbb{F}_{q}}$ to be generated by elements $\phi^{-1}(\mathfrak{m}_{x,y})^{i}d\hat{\mathcal{O}}_{X,x}$ and $\mathfrak{m}_{x,y}^{j}d\mathcal{O}_{x,y}$ respectively. Then define

\[\Omega^{1,cts}_{F_{x,y}/\mathbb{F}_{q}} : = \Omega^{1}_{F_{x,y}/\mathbb{F}_{q}}/F_{x,y} . \cap_{i,j \geq 0} (P_{i}+T_{j}),\]
and \[\Omega^{2, cts}_{F_{x,y}/\mathbb{F}_{q}} : = \Omega^{1, cts}_{F_{x,y}/\mathbb{F}_{q}} \wedge \Omega^{1, cts}_{F_{x,y}/\mathbb{F}_{q}}.\]
\\

Now we recall the definition of the residue homomorphism. 
\begin{defn}
Let $F_{x,z}$ a two-dimensional local field of positive characteristic, and fix an isomorphism $F_{x,z} \cong k_{z}(x)((t_{1}))((t_{2}))$. Define the residue homomorphism
\[res_{F_{x,z}}: \Omega^{2, cts}_{F_{x,z}/k_{z}(x)} \to \mathbb{F}_{q}\]
by $res_{F_{x,z}}(\omega) = \text{Tr}_{k_{z}(x)/\mathbb{F}_{q}}a_{-1,-1}$ where
\[\omega = \sum a_{a_{1},a_{2}}t_{1}^{a_{1}}t_{2}^{a_{2}}dt_{1}\wedge dt_{2}.\]
\end{defn}
The residue map is independent of the choice of local parameters $t_{1}$ and $t_{2}$, see \cite{P1} section one.\\
Now let $A$ be the fraction field of the ring of Witt vectors of $\mathbb{F}_{q}$ and $L = A((t_{1}))((t_{2}))$. This lift to characteristic zero is necessary to define the following auxiliary co-ordinates and polynomials, but notice that in the end the formulae will be 'denominator free', so the reduction back down to positive characteristic is well-defined.\\
Let $x=(x_{0}, x_{1}, \dots ) \in L$, and for each $m \in \mathbb{Z}$ introduce the auxiliary co-ordinates
\[x(m) = x_{0}^{p^{m}} + px_{1}^{p^{m-1}} + \dots + p^{m}x_{m}\]
and the polynomials $P_{m}(X_{0}, X_{1}, \dots , X_{m}) \in \mathbb{Z}[p^{-1}][X_{0}][X_{1}]\dots [X_{m}]$ such that $P_{m}(x(0), x(1), \dots, x(m)) = x_{m}$.
\begin{defn}
Let $f_{1}$, $f_{2} \in F_{x,z}^{\times}$, $g \in W(F_{x,z})$ and $\bar{g} \in W(L)$ an element such that $\bar{g}$ mod $p$ $= g$. Define the Witt pairing by
\[ (f_{1}, f_{2}|g]_{x,z} = (\text{Tr}_{\mathbb{F}_{q}/\mathbb{F}_{p}}w_{i})_{i \geq 0} \in W(\mathbb{F}_{p})\]
where for each $i \in \mathbb{Z}$, 
\[w_{i} = P_{i}\left(res_{L}\left(\bar{g}(0)\frac{df_{1}}{f_{1}}\wedge \frac{df_{2}}{f_{2}}\right), \dots, res_{L}\left(\bar{g}(i)\frac{df_{1}}{f_{1}}\wedge \frac{df_{2}}{f_{2}}\right)\right) \text{ mod } p \] where the $\bar{g}(j)$ are the auxiliary polynomials for the Witt vector $\bar{g}$. Then as before for a curve $y$ with branches $z$, define
\[( \ , \ | \ ]_{x,y} = \sum_{z \in y(x)}( \ , \ | \ ]_{x,z}.\]
\end{defn}

\begin{prop}\label{wittproperties}
The Witt pairing satisfies the following properties:
\begin{enumerate}
\item{$(f_{1}.f_{1}', f_{2}|g]_{x,y} = (f_{1},f_{2}|g]_{x,y} + (f_{1}',f_{2}|g]_{x,y}$ and $(f_{1},f_{2}.f_{2}'|g]_{x,y} = (f_{1},f_{2}|g]_{x,y} + (f_{1},f_{2}'|g]_{x,y}$;}
\item{$(f_{1},f_{2}|g+h]_{x,y} = (f_{1},f_{2}|g]_{x,y} + (f_{1},f_{2}|h]_{x,y}$;}
\item{$(f_{1}, 1-f_{1}|g]_{x,y} = 0$;}
\item{$(f_{1},f_{2}|g]_{x,y} = (w_{0}, w_{1}, \dots) \implies (f_{1},f_{2}|g^{p}]_{x,y} = (w_{0}^{p}, w_{1}^{p}, \dots);$}
\item{$(f_{1},f_{2}|g]_{x,y}$ is continuous in each argument;}
\item{$(f_{1},f_{2}| g_{0},  \dots, g_{m-1}]_{x,y} = (w_{0}, \dots, w_{m-1}) \implies (f_{1}, f_{2}|g_{0}, \dots , g_{m-2}]_{x,y} = (w_{0}, \dots, w_{m-2});$}
\item{$(f_{1}, f_{2}| 0, g_{1}, \dots , g_{m-1}]_{x,y} = (0, (f_{1}, f_{2}| g_{1}, \dots, g_{m-1}]_{x,y})$.}
\end{enumerate}\end{prop}
\begin{proof} In \cite{P4}, 3.3.6, Parshin proves this for a single higher local field. We will prove it here for the case where $x$ is a singular point of $y$ and we must sum the pairings over each branch of $y$ at $x$. \\
Property 3 follows straight away, and properties 1 and 2 follow from the fact that trace distributes over addition.\\
Property 4 is true as
\[(f_{1},f_{2}|g^{p}]_{x,y} = \sum_{z \in y(x)} (f_{1},f_{2}|g^{p}]_{x,z} = \sum_{z \in y(x)} (w_{0,x,z}^{p},w_{1,x,z}^{p}, \dots) \]\[= \left( \sum_{z \in y(x)} w_{0,x,z}^{p}, \dots\right) = \left( \left(\sum_{z \in y(x)}w_{0,x,z}\right)^{p}, \dots \right) = (w_{0}^{p}, w_{1}^{p}, \dots)\]
where equality holds as the sum of Witt vectors is given by polynomials in their coeffiicients, and when taking powers of $p$ we just raise each coefficient to the power $p$.\\
Property 5 follows from the continuity of trace and addition. 7 is true because when summing Witt vectors, the $n^{th}$ term depends linearly only on the $0^{th}, \dots (n-1)^{th}$ terms of the vectors being summed: so if the $0^{th}$ term is $0$ for all $z \in y(x)$ then it will be in the sum also.\\
Finally, property 6 follows straight from \cite{P4}, and the fact that Witt vector summation depends only on lower terms as mentioned above.

\end{proof}

Properties one and three show that the Witt symbol is a symbol on $K_{2}(F_{x,y})$ in the first two arguments.

\section{Relation to Higher Class Field Theory}
Both the higher tame symbol and the Witt symbol are related to higher local class field theory. The Witt symbol was first used by Kawada and Satake to get a reciprocity map in the wildly ramified case for local and global fields of positive characteristic in \cite{KS}. In \cite{P4}, Parshin used both symbols to define the ramified part of the reciprocity map for higher local fields, via some well-known dualities. We will outline these dualities below, but for full details see \cite{P4}. The higher tame symbol and the Witt symbol are both sequentially continuous, so can be seen as symbols on $K_{2}^{top}(F_{x,y})$.\\
Firstly, the higher tame symbol is related to the Kummer extensions $\mathbb{F}_{q}((t_{1}^{1/l}))((t_{2}^{1/l}))/F_{x,y}$ where $l$ divides $q-1$. The higher tame symbol is a non-degenerate pairing on
\[K_{2}^{top}(F_{x,y})/(q-1)K_{2}^{top}(F_{x,y}) \times F_{x,y}^{\times}/(F_{x,y}^{\times})^{q-1} \to \mathbb{F}_{q}^{\times}\]
inducing an isomorphism
\[K_{2}^{top}(F_{x,y})/(q-1)K_{2}^{top}(F_{x,y}) \cong G_{1}\]
following from Kummer duality, where $G_{1}$ is the Galois group of the extension $\mathbb{F}_{q}((t_{1}^{1/(q-1)}))((t_{2}^{1/(q-1)}))/F_{x,y}$.\\
Similarly, the Witt symbol is related via Witt duality to the Galois group of all $p$-extensions of $F_{x,y}$. The method of using such a pairing to exploit this duality was first applied to one-dimensional local and global fields by Kawada and Satake in their paper \cite{KS}. Taking the Witt symbol as a pairing on the topological $K_{2}$-group of $F_{x,y}$ and the Witt vectors of $F_{x,y}$, we get a pairing
\[K_{2}^{top}(F_{x,y})/K_{2}^{top}(F_{x,y})_{tors} \times W(F_{x,y})/(\text{Frob}-1)W(F_{x,y}) \to \mathbb{Z}_{p}\]
where $\mathbb{Z}_{p} = W(\mathbb{F}_{p})$. The pairing induces an isomorphism
\[K_{2}^{top}(F_{x,y})/K_{2}^{top}(F_{x,y})_{tors} \cong \text{Hom}(W(F_{x,y})/(\text{Frob}-1)W(F_{x,y}), \mathbb{Z}_{p}) \cong G^{ab,p}\]
where $G^{ab,p}$ is the Galois group of the maximal abelian $p$-extension of $F_{x,y}$ and the second isomorphism follows from Witt duality.\\
This is the setting in which we will prove reciprocity laws for the symbols - viewing them both as taking values in the Galois group allows us to prove reciprocity laws which work for both symbols at the same time, in a way we could not if just viewing them as described above, as the higher tame symbol takes values in the \emph{multiplicative} group $\mathbb{F}_{q}^{\times}$ and the Witt symbol takes values in the \emph{additive} group $W(\mathbb{F}_{p})$.\\
We will now examine how to view these pairings when summing about a point or along a curve.\\
 First, fix a closed point $x \in X$. Let $F_{x}$ be the ring generated by $\hat{\mathcal{O}}_{X,x}$ and $F$. Let $L/F_{x}$ be a finite \'{e}tale extension with Galois group $G$ (see \cite{Mi} or \cite{GT}), $\mathcal{O}_{L}$ the integral closure of $\hat{\mathcal{O}}_{X,x}$ in $L$ and $\mathfrak{p}_{L}$ its maximal ideal. As mentioned at the start of section two, every height one prime ideal $\mathfrak{q} \subset \mathfrak{p}_{L}$ determines a two-dimensional local field $L_{\mathfrak{p}_{L}, \mathfrak{q}}$. Spec$(\mathcal{O}_{L})$ is a normal two-dimensional scheme over the residue field $l$, a finite extension of $k(x)$, and we have a finite morphism
\[\phi: \text{Spec}(\mathcal{O}_{L}) \to \text{Spec}(\hat{\mathcal{O}}_{X,x}).\] 
For a height one prime ideal $\mathfrak{q}$ of $\mathcal{O}_{L}$, define the stabiliser
\[G_{\mathfrak{q}} = \{g \in G : g(\mathfrak{q}) = \mathfrak{q}\}.\]
If $\mathfrak{q}$, $\mathfrak{q}'$ are two such primes, and $\phi(\mathfrak{q}) = \phi(\mathfrak{q}')$ then $G_{\mathfrak{q}}$ is conjugate to $G_{\mathfrak{q}'}$ in $G$. \\
Now let $L/F_{x}$ be an abelian extension - then the homomorphism
\[\text{Gal}(L_{\mathfrak{p}_{L}, \mathfrak{q}}/F_{x,y}) \cong G_{\mathfrak{q}} \to G = \text{Gal}(L/F_{x})\]
is independent of the choice of $\mathfrak{q}$, where $\mathfrak{q}$ is any prime ideal such that $\phi(\mathfrak{q})$ is the prime ideal of $\hat{\mathcal{O}}_{X,x}$ associated to the curve $y$. Of course, this (and the statement below for curves) is just basic valuation theory - see \cite{Bour}, chapter VI.\\ 
So the product of all the symbols
\[\prod_{y \ni x}K_{2}^{top}(F_{x,y}) \to \text{Gal}(F_{x}^{ab}/F_{x})\]
is well-defined.
It remains to check this product of symbols converges by proving that for $f$, $g$ and $h \in F$, $h' \in W(F)$, both symbols $(f,g,h)_{x,y}$ and $(f,g|h']_{x,y}$ are trivial for all but finitely many $y \ni x$ - this will be done in the following section.\\
Now fix a reduced irreducible curve $y \subset X$. Our method is very similar to the one above for a fixed point, using only basic Galois theory and decomposition groups. Let $F_{y}$ be the field of fractions of $\mathcal{O}_{y}$ - this will be a discrete valuation field with residue field a global field of positive characteristic. Let $L/F_{y}$ be a finite Galois extension with Galois group $G$ - then $L$ will also be a complete discrete valuation field over a global field. The extension of residue fields $\bar{L}/k(y)$ determines a finite morphism of curves $\pi: y' \to y$, where $y' = y \times_{F} L$ and  $k(y') = \bar{L}$. For each point $x \in y$, we have the decomposition
\[L \otimes_{F_{y}} F_{x,y} = \oplus_{x' \in y', \pi(x') = x}L_{x', y'}\]
where the $L_{x',y'}$ are products of two-dimensional local fields. Each term in the product is a finite extensions of $F_{x,z}$, where $z$ is a branch of $y$ passing through $x$.   \\

 For each $x' \in y'$ with $\pi(x')=x$, define
\[G_{x'}=\{g \in G: g(x')=x'\}.\]
As before, for another $x''$ such that $\pi(x'')=x$, the groups $G_{x'}$ and $G_{x''}$ are conjugate in $G$. We have $G_{x'} \cong$ Gal$(L_{x',y'}/F_{x,y})$.\\
Now let $L/F_{y}$ be an abelian extension. Then the homomorphism 
\[\text{Gal}(L_{x',y'}/F_{x,y}) \cong G_{x'} \to G = \text{Gal}(L/F_{y})\]
is independent of the choice of $x'$. \\
So the product of the symbols
\[\prod_{x \in y}K_{2}^{top}(F_{x,y}) \to \text{Gal}(F_{y}^{ab}/F_{y})\] is well-defined.
Again, we must check that this converges in the group Gal$(F_{y}^{ab}/F_{y})$ - see the section on reciprocity for curves below.\\
This is the context in which we will prove reciprocity laws for the symbols - viewing their values as elements of the Galois group of the maximal abelian extension of the fields associated to the point $x$ and the curve $y$.

\section{Reciprocity at a Point}
In this section we will prove a reciprocity law for the gluing together of the higher tame symbol and the Witt symbol in the Galois group Gal$(F_{x}^{ab}/F_{x})$. Our first step is to prove the lemma promised above, i.e. that the sum converges in Gal$(F_{x}^{ab}/F_{x})$ - it is simpler to prove that the value of each symbol is trivial in $\mathbb{F}_{q}^{\times}$ or $W(\mathbb{F}_{p})$ for almost all $y \ni x$ rather than working in the Galois group at this stage. We will keep the pairings separate to avoid confusion.
 \begin{lemma}\label{wittzeros}Let $f, g \in F$, $h \in W(F)$ and fix a point $x \in X$. Then for all but finitely many $y$ passing through $x$, the Witt symbol $(f,g|h]_{x,y}$ is zero.\end{lemma}
 \begin{proof}
 For each $y \ni x$ and $z \in y(x)$, fix a local parameter of the curve $t_{z} \in \mathcal{O}_{X,z}$ and a local parameter at $x$, $u_{x,z} \in k(z)$. Then
  \[F_{x,z}^{\times} \cong k_{z}(x)^{\times} \times \langle u_{x,z} \rangle \times \langle t_{z} \rangle \times \mathcal{U}_{x,z}\]
 where $\mathcal{U}_{x,z}$ is the group of principal units in $\mathcal{O}_{x,y}^{\times}$.\\
 Then $f,g$ and each entry $h_{i}$ can be expanded in each $F_{x,z}^{\times}$ as, e.g. $f = \alpha_{x,z} u_{x,z}^{i} t_{z}^{j} \varepsilon_{x,z}$ with $\alpha_{x,z} \in k_{z}(x)^{\times}$, $i,j \in \mathbb{Z}$ and $\varepsilon_{x,z} \in \mathcal{U}_{x,z}$. Then for all but finitely many $y \ni x$, the exponent of $t_{z}$ will be zero - this is because each $t_{z}$ represents an irreducible polynomial in $F_{x}$, and any element can be divisible by only finitely many of these.\\
 So for all but finitely many $y$ with $z \in y(x)$, we have $f,g$ and $P_{i}(h_{0},\dots,h_{i-1)} \in \mathcal{O}_{x,z}$. The following lemma will complete the proof.
 \end{proof}
\begin{lemma}
Let $F_{x,z}$ be a two-dimensional local field over $\mathbb{F}_{q}$. The residue map satisfies
\[\text{res}_{F_{x,z}}(\omega) = 0 \text{ for all }\omega \in \Omega^{2, cts}_{\mathcal{O}_{F_{x,z}}/\mathbb{F}_{q}}.\]\end{lemma}
\begin{proof}
Fix an isomorphism $F_{x,z} \cong \mathbb{F}_{q}((t_{1}))((t_{2}))$ and let $f \in \mathcal{O}_{F_{x,z}}$. Similarly to lemma 2.8 in \cite{Mor3}, we may write $f = \sum_{i,j=0}^{n}a_{i,j}t_{1}^{i}t_{2}^{j} + gt_{1}^{n+1}t_{2}^{n+1}$ for any integer $n$ and some $g \in \mathcal{O}_{F_{x,z}}$. Applying the universal derivation $d: \mathcal{O}_{F_{x,z}}\to \Omega_{\mathcal{O}_{F_{x,z}}/\mathbb{F}_{q}}^{1}
$, we have
\[ df = \sum_{i,j=0}^{n} a_{i,j}(it_{1}^{i-1}t_{2}^{j}dt_{1}+jt_{1}^{i}t_{2}^{j-1}dt_{2}) + g(n+1)(t_{1}^{n}t_{2}^{n+1}dt_{1}+ t_{1}^{n+1}t_{2}^{n}dt_{2}) + t_{1}^{n+1}t_{2}^{n+1}dg.\]
Hence $df - \left(\frac{df}{dt_{1}}dt_{1} + \frac{df}{dt_{2}}dt_{2}\right) \in \cap_{n=1}^{\infty} t_{1}^{n}t_{2}^{n}\Omega_{\mathcal{O}_{F_{x,z}}/\mathbb{F}_{q}}^{1}$. So taking the separated quotient, $\Omega_{\mathcal{O}_{F_{x,z}}/\mathbb{F}_{q}}^{1}$ is generated by $dt_{1}$ and $dt_{2}$. Then $\Omega_{\mathcal{O}_{F_{x,z}}/\mathbb{F}_{q}}^{2} = \Lambda^{2}\Omega_{\mathcal{O}_{F_{x,z}}/\mathbb{F}_{q}}^{1, cts}$ is generated over $\mathcal{O}_{F_{x,z}}$ by $dt_{1}\wedge dt_{2}$, as all other terms in the exterior product are zero.\\
Hence we can restrict to the case $\omega = adt_{1}\wedge dt_{2}$ where $a \in \mathcal{O}_{F_{x,z}}$ and $t_{1}$ and $t_{2}$ are the local parameters of $F_{x,z}$. Decomposing $a$ as a series
\[a = \sum_{i \geq I}\sum_{j \geq 0}a_{i,j}t_{1}^{i}t_{2}^{j}\]
gives the result.\end{proof}
Now we move on to the higher tame symbol.
\begin{lemma}Let $f,g,h \in F^{\times}$ and fix a point $x \in X$. Then for all but finitely many curves $y$ passing through $x$, the higher tame symbol $(f,g,h)_{x,y}$ is equal to one.\end{lemma}
\begin{proof}
We proceed in a similar fashion to the proof for the Witt symbol. Again, we may decompose our elements as products $\alpha_{x,y}u_{x,y}^{i}t_{y}^{j}\varepsilon_{x,y}$, with $j=0$ for all but finitely many $y$. If we fix a $j \in \mathbb{Z}$, then there are only finitely many expansions of a fixed element with the exponent of $t_{y}$ being $j$ and the exponent of $u_{x,y}$ being non-zero. So the number of $y$ with $i$ and $j$ not equal to zero is certainly finite. So for all but finitely many $y$, $f,g$ and $h$ will all be in the group $k(x)^{\times} \times \mathcal{U}_{x,y}$.\\
Basic properties of the higher tame symbol show it is trivial if any entry is in the group of principal units $\mathcal{U}_{x,y}$, and a simple calculation shows it is also trivial if more than one of the entries is in $k(x)^{\times}$, which completes the proof.\end{proof}

The triviality of the symbols in these groups of course proves the triviality of the symbols after the homomorphism to the Galois group. We now discuss reciprocity laws in the same way - keeping the symbols separate and in $W(\mathbb{F}_{p})$ and $\mathbb{F}_{q}{^\times}$, proving that the sums and products over all $y \ni x$ are trivial, then observing this implies the same in the abelian Galois group to glue the two symbols together.

\begin{thm}\label{reciprocitywitt}
Let $f,g \in F^{\times}$, $h \in W(F)$.  Then the sum
\[\sum_{y \ni x}(f,g|h]_{x,y} = 0\] 
in the ring of Witt vectors $W(\mathbb{F}_{p})$.\end{thm}

We will proceed in a very similar manner to Morrow in \cite{Mor3}, section 3, which deals with reciprocity laws for residues on an arithmetic surface.  We begin with a positive characteristic analogue of Morrow's lemma 3.4.

\begin{lemma}\label{addstructure}
Let $F_{x}$ be the field associated to $x \in X$ as above, $\hat{\mathcal{O}}_{X,x} \cong  k(x)((t_{1}))((t_{2}))$. Then every element of $F_{x}$ is a finite sum of elements of the form
\[\frac{\beta}{\alpha^{m}\gamma^{n}}\]
where $\alpha, \gamma$ are distinct irreducible polynomials in $k(x)(t_{1})(t_{2})$, $\beta \in \hat{\mathcal{O}}_{X,x}$ and $m,n \geq 0$ with at least one greater than zero.\end{lemma}
\begin{proof}
Following \cite{Mor3}, we begin with an element of the form $1/\pi_{1}^{n_{1}}\pi_{2}^{n_{2}}\pi_{3}^{n_{3}}$ with the $\pi_{i}$ distinct irreducible elements of $\hat{\mathcal{O}}_{x,x}$ and the $n_{i} \geq 1$. Then we have $1 = \beta_{1}\pi_{1}^{n_{1}}\pi_{2}^{n_{2}} + \beta_{2}\pi_{3}^{n_{3}}$ for some $\beta_{1}, \beta_{2} \in \hat{\mathcal{O}}_{X,x}$ and hence
\[\frac{1}{\pi_{1}^{n_{1}}\pi_{2}^{n_{2}}\pi_{3}^{n_{3}}} = \frac{\beta_{1}}{\pi_{3}^{n_{3}}} + \frac{\beta_{2}}{\pi_{1}^{n_{1}}\pi_{2}^{n_{2}}}.\]

For the general case, we write any element of $F_{x}$ as $a/b$ with $a,b \in \hat{\mathcal{O}}_{X,x}$. We may uniquely factorise $b$ as $u\gamma_{1}^{r_{1}}\dots\gamma_{s}^{r_{s}}$ where $u \in \hat{\mathcal{O}}_{X,x}^{\times}$, the $\gamma_{i}$ are irreducible polynomials and all exponents are $> 0$. We may replace $a$ with $u^{-1}a$ and so suppose $u=1$. Then repeated application of the first part of the proof decomposes $a/b$ into a sum as required.\end{proof}

We now proceed to the proof of \ref{reciprocitywitt}.
\begin{proof}
We first prove that if $\hat{\mathcal{O}}_{X,x}$ is regular, then for each $\omega \in \Omega^{2, cts}_{F/\mathbb{F}_{q}}$, $\sum_{y \ni x} \text{Tr}_{\mathbb{F}_{q}/\mathbb{F}_{p}}res_{x,y}(\omega) = 0$.
By lemma \ref{addstructure}, it is enough to reduce to the case 
\[\omega = \frac{\beta}{\alpha^{m}\gamma^{n}}dt_{1}\wedge dt_{2}\]
where $ \beta, \alpha, \gamma, m, n$ are all as in lemma \ref{addstructure} and $\hat{\mathcal{O}}_{X,x} \cong k(x)[[t_{1}]][[t_{2}]]$.\\
Let $y$ be a reduced irreducible curve containing $x$. Then $y$ is associated to an irreducible polynomial $t_{y} \in \hat{\mathcal{O}}_{X,x}$. Fix $u_{x,y}$ a local parameter of $x$ at $y$ such that $\langle t_{1}, t_{2}\rangle = \langle u_{x,y}, t_{y}\rangle$. Suppose first that $t_{y}$ is not equal to either $\alpha$ or $\gamma$. Then $\beta/\alpha^{m}\gamma^{n}$, $t_{1}$ and $t_{2}$ all belong to the completion of $\hat{\mathcal{O}}_{X,x}$ at the prime ideal $t_{y}\hat{\mathcal{O}}_{X,x}$, i.e. $\mathcal{O}_{F_{x,y}}$. Then $t_{y}^{-1}$ cannot have a non-zero coefficient in 
\[\frac{\beta}{\alpha^{m}\gamma^{n}}\left(\frac{ dt_{1}}{du_{x,y}} \wedge \frac{dt_{2}}{dt_{y}} + \frac{dt_{1}}{dt_{y}} \wedge \frac{dt_{2}}{du_{x,y}}\right)\] and so 

\[\text{res}_{x,y}(\omega) = \text{coeff}_{u_{x,y}^{-1}t_{y}^{-1}}\left(\frac{\beta}{\alpha^{m}\gamma^{n}}\left(\frac{ dt_{1}}{du_{x,y}} \wedge \frac{dt_{2}}{dt_{y}} + \frac{dt_{1}}{dt_{y}} \wedge \frac{dt_{2}}{du_{x,y}}\right)du_{x,y}\wedge dt_{y}\right) = 0.\]

Similarly to \cite{Mor3}, theorem 3.6, this reduces us to proving that
\[\text{Tr}_{\mathbb{F}_{q}/\mathbb{F}_{p}}\text{res}_{x, \alpha}(\omega) +\text{Tr}_{\mathbb{F}_{q}/\mathbb{F}_{p}}\ \text{res}_{x, \gamma}(\omega) = 0.\]
This part of our proof is much easier that Morrow's case, as the bulk of his proof involves passing between the different characteristics found in residue fields of a arithmetic surface. \\
We must use an important property of the higher residue symbol: let $K$ be a higher local field with final local parameter $t$. The for any $x \in \mathcal{O}_{K}^{\times}$, 
\[\text{res}_{K}\left(\frac{dx}{x}\wedge\frac{dt}{t}\right) = \text{res}_{\bar{K}}\left(\frac{d\bar{x}}{\bar{x}}\right).\]
To see this, first let $x = 1+at$, some $a \in \mathcal{O}_{K}$. Then
\[\frac{dx}{x} \wedge \frac{dt}{t} = x^{-1}da\wedge dt \in \Omega^{2, cts}_{\mathcal{O}_{K}/\mathbb{F}_{q}}\]
and so its residue is zero - but res$_{\bar{K}}(d\bar{x}/\bar{x}) = 0$ also, so we are done in this case. \\
The symbol $dt/t$ is additive with respect to multiplication by $t$, so we can now restrict to the case $\bar{x} \in \bar{F}^{\times}$, $x = \bar{x} + bt$ with $b \in \mathcal{O}_{K}$. Then
\[\text{res}_{K}\left(\frac{dx}{x}\wedge\frac{dt}{t}\right) = \text{res}_{K}\left( \frac{d(\bar{x}+bt)}{\bar{x}+bt}\wedge\frac{dt}{t}\right) = \text{res}_{\bar{K}}\left(\frac{d\bar{x}}{\bar{x}}\right)\]
by expanding $(\bar{x} + bt)^{-1}$.\\
Now take $y$ to be the curve associated to the prime ideal $\alpha\mathcal{O}_{X,x}$.  Then we may take $\alpha$ to be the local parameter for $y$ in $F_{x,y}$.
Expand $\beta$ as a sum $\sum_{i,j}b_{i,j}\gamma^{i}\alpha^{j}$, $b_{i,j} \in \mathbb{F}_{q}$ - we may assume that $\alpha$ and $\gamma$ generate the maximal ideal of $\hat{\mathcal{O}}_{X,x}$, as otherwise the residue is certainly zero. Then

\[\text{res}_{x,\alpha}\left(\frac{\sum_{i,j}b_{i,j}\gamma^{i}\alpha^{j}}{\alpha^{m}\gamma^{n}} \left(\frac{ dt_{1}}{du_{x,y}} \wedge \frac{dt_{2}}{d\alpha} + \frac{dt_{1}}{d\alpha} \wedge \frac{dt_{2}}{du_{x,y}}\right)du_{x,y}\wedge d\alpha\right)\]
 \[= \text{res}_{x}\left(\overline{\left(\frac{\sum_{i}b_{i,m-1}\gamma^{i}}{\gamma^{n}} \right)}\left(\frac{dt_{1}}{du_{x,y}} + \frac{dt_{2}}{du_{x,y}}\right)du_{x,y}\right).\]

Then apply exactly the same argument for $x$ as for the curve $y$ - the only prime of $k(y)$ for which this residue can be non-zero is $\bar{\gamma}$.   Again we may take $\bar{\gamma}$ to be the local parameter at $x$, and see that the residue equals $b_{n-1,m-1}$.
   
  Following the same argument, but first taking $y$ to be the curve defined by $\gamma$, we see the residue in this case will be $-b_{n-1,m-1}$, the negation coming from the equality
  \[\text{res}_{x,\gamma}\left(\frac{\sum_{i,j}b_{i,j}\gamma^{i}\alpha^{j}}{\alpha^{m}\gamma^{n}} \left(\frac{ dt_{1}}{d\gamma} \wedge \frac{dt_{2}}{d\alpha} + \frac{dt_{1}}{d\gamma} \wedge \frac{dt_{2}}{du_{x,y}}\right)d\gamma\wedge d\alpha\right)\]
  \[= - \text{res}_{x,\gamma}\left(\frac{\sum_{i,j}b_{i,j}\gamma^{i}\alpha^{j}}{\alpha^{m}\gamma^{n}} \left(\frac{ dt_{1}}{d\gamma} \wedge \frac{dt_{2}}{d\alpha} + \frac{dt_{1}}{d\gamma} \wedge \frac{dt_{2}}{du_{x,y}}\right)d\alpha\wedge d\gamma\right).\]

So we have proved the reciprocity law in the regular case.\\
\\
Now we suppose $\hat{\mathcal{O}}_{X,x}$ is a two-dimensional, normal complete local ring of characteristic $p$ with finite final residue field $k(x)$ - these assumptions match the assumption that $X$ is a normal projective surface over a finite field $\mathbb{F}_{q}$. It is easy to prove a positive characteristic version of Morrow's lemma 3.7 in \cite{Mor3}: by \cite{Coh}, theorem 16, $\hat{\mathcal{O}}_{X,x}$ contains a subring $B_{0}$ such that $\mathcal{O}_{X,x}$ is a finite $B_{0}$-module and $B_{0}$ is a two-dimensional local ring with residue field $k(x)$. Set $i: \mathbb{F}_{q}[[t_{1}]][[t_{2}]] \to B_{0}$ an isomorphism. Then to show $\hat{\mathcal{O}}_{X,x}$ is a finite $B_{0}$-module, we proceed exactly as in Morrow's proof.
So there is a ring $B$ between $\hat{\mathcal{O}}_{X,x}$ and $\mathbb{F}_{q}$ such that $B \cong k(x)[[t_{1}]][[t_{2}]]$ and $\hat{\mathcal{O}}_{X,x}$ is a finite $B$-module.\\
The usual local-global trace formula tells us that for $\omega \in \Omega^{2, cts}_{\hat{\mathcal{O}}_{X,x}/\mathbb{F}_{q}}$ and $y \ni x$,
\[\text{Tr}_{F/\text{Frac}(B)}(\omega) = \sum_{Y|y} \text{Tr}_{F_{Y}/\text{Frac}(B_{y})}(\omega).\]
Then using the functoriality of the residue map - see \cite{P4}, lemma 3.4.4 - we have
\[\text{res}_{y}\text{Tr}_{F/\text{Frac}(B)}(\omega) = \sum_{Y|y}\text{res}_{y}\text{Tr}_{F_{Y}/\text{Frac}(B_{y})}(\omega) = \sum_{Y|y} \text{Tr}_{k(x)/\mathbb{F}_{q}}\text{res}_{Y}(\omega).\]
Finally applying $\text{Tr}_{\mathbb{F}_{q}/\mathbb{F}_{p}}$ to both sides yields
\[\text{Tr}_{\mathbb{F}_{q}/\mathbb{F}_{p}}\text{res}_{y}\text{Tr}_{F/\text{Frac}(B)}(\omega) = \sum_{Y|y} \text{Tr}_{k(x)/\mathbb{F}_{p}}\text{res}_{Y}(\omega).\]
Now we complete the proof in the case where $X$ is not assumed to be regular. Let $x'$ be the image of $x$ in $B$. We have
\[ \sum_{Y \ni x} \text{Tr}_{k(x)/\mathbb{F}_{p}}\text{res}_{Y} (\omega) = \sum_{y \ni x'}\sum_{Y|y}\text{Tr}_{k(x)/\mathbb{F}_{p}}(\text{res}_{Y}\omega) 
= \sum_{y \ni x'} \text{Tr}_{\mathbb{F}_{q}/\mathbb{F}_{p}}\text{res}_{y}(\text{Tr}_{k(x)/\mathbb{F}_{q}}(\omega)) = 0\]
by the reciprocity law for regular rings proved above.\\
The case for a general Witt vector follows by induction. The case for a one-dimensional Witt vector is done above, so we suppose that the sum is zero for all $f$, $g \in F$ and $h \in W_{m-1}(F)$, the ring of truncated Witt vectors of length $m-1$. Let $h  = (h_{0}, \dots, h_{m-1}) \in W_{m}(F)$. By property 6 in proposition \ref{wittproperties}, 
\[(f,g|(h_{0}, \dots, h_{m-1})]_{x,y} = (w_{0}, \dots, w_{m-1})\]
implies
\[(f,g|(h_{0}, \dots, h_{m-2})]_{x,y} = (w_{0}, \dots, w_{m-2}).\]
Suppose there exists $h \in W_{m}(F)$ such that $\sum_{y \ni x}(f,g|h]_{x,y} \neq 0 $ for $f$, $g \in F$, i.e
\[\sum_{y \ni x}(f,g|(h_{0}, \dots, h_{m-1})]_{x,y} \neq 0.\]
But by our induction hypothesis, 
\[\sum_{y \ni x}(f,g|(h_{0}, \dots, h_{m-2})]_{x,y} = 0,\]
so the only non-zero term in $(w_{0}, \dots, w_{m-1})$ can be the last entry. By the definition of the Witt pairing, this means that $h_{m-1}$ can be the only non-zero term of the Witt vector $h$. Now using property 7 from proposition \ref{wittproperties}, we can reduce back to the one-dimensional case for a contradiction.
\end{proof}

\begin{thm}
Let $f,g,h \in F^{\times}$. Then the product
\[\prod_{y \ni x} \prod_{z \in y(x)}N_{k_{z}(x)/\mathbb{F}_{q}}(f,g,h)_{x,z} = 1\]
in the group $\mathbb{F}_{q}^{\times}$.\end{thm}

There have been many proofs of this proposition, however some neglect the norm from $k_{z}(x)$ to $k(x)$ and so are not quite correct. We offer a corrected explicit proof below.

\begin{proof}
We use the decomposition
  \[F_{x,z}^{\times} \cong k_{z}(x)^{\times} \times \langle t_{1} \rangle \times \langle t_{2} \rangle \times \mathcal{U}_{x,z},\]
where $t_{1}$ and $t_{2}$ are any two local parameters generating the maximal ideal at $x$, to look at several different cases.

Firstly, it is easy to see from the definition and the discussion at the beginning of section four that none of the entries can be from $\mathcal{U}_{x,z}$ and at least two must not be constants in $k_{z}(x)$.

If two of the entries are a prime $t$, without loss of generality we assume they are the entries in the group $K_{2}^{top}(F_{x,z})$ - i.e. the first two entries. Then we use the $K$-group relation $\{t,t\} = \{t,-1\}$ to return to the above case.

So we are left with the case where either the first two entries are local equations of different curves passing though $x$ and the last is a constant, or all three entries are different equations of curves.

For the first of these case, we let the two curves whose equations appear in the symbol be $z_{1}$ and $z_{2}$ and compute
\[ \prod_{y \ni x}\prod _{z \in y(x)} N_{k_{z}(x)/\mathbb{F}_{q}}(f,g,h)_{x,z} = N_{k_{z_{1}}(x)/\mathbb{F}_{q}}(f,g,h)_{x,z_{1}}N_{k_{z_{2}}(x)/\mathbb{F}_{q}}(f,g,h)_{x,z_{2}}
\]
as the symbol is trivial whereever $f$, $g$ and $h$ are not local parameters of the curve related to the field. This is then equal to

\[N_{k_{z_{1}}(x)/\mathbb{F}_{q}}(h^{\bar{v}_{g}(f)}  \text{ mod }\mathfrak{p}_{f,g})N_{k_{z_{2}}(x)/\mathbb{F}_{q}}(h^{- \bar{v}_{f}(g)}\text{ mod }\mathfrak{p}_{f,g}) \]
where the $\bar{v}_{f}$ and $\bar{v}_{g}$ appear as we may take the equations $f$ and $g$ to be the other local parameters in the fields $F_{x,z_{2}}$ and $F_{x,z_{1}}$ respectively. A final computation reveals this is equal to
\[ (h \text{ mod } \mathfrak{p}_{f,g})^{(f,g)_{x} - (f,g)_{x}} = 1\]
where $(f,g)_{x} = $ dim$_{\mathbb{F}_{q}}(\mathcal{O}_{X,x}/(f,g))$ is the local intersection number at $x$.

Finally, we have the case where $f$, $g$ and $h$ are all equations of different curves passing through $x$. We can use blow-ups to reduce to the case where $x$ lies on only two curves - see \cite{AG}, V.3 - but we must check this doesn't change the value of the higher tame symbol.

Let $z$ be a branch of a curve $y$ at $x$ and $z'$ its preimage after the blowing-up $\phi$, $x'$ the preimage of $x$ and $f'$, $g'$ and $h'$ the preimages of $f$, $g$ and $h$. We wish to show that
\[(f,g,h)_{x,z} = (f',g',h')_{x',z'}.\]
We take local parameters $t_{1}$, $t_{2}$ at $x$ and $s_{1}$, $s_{2}$ at $x'$ such that $\phi^*(t_{1}) = s_{1}$ and $\phi^*(t_{2}) = s_{1}s_{2}$ - see \cite{AG}, I.4. 

Now we can proceed to calculate both symbols using Fesenko's determinant method - it is easy to see using these local parameters that we have just 
replaced our matrix by one with a column of the matrix added to another, and thus the determinant remains unchanged. So we have returned to the above case and the proof is complete.
\end{proof}

We mention some other methods of proving this theorem, all of which take singular points into account. A recent approach uses central extensions to construct the higher tame symbol as an analogue of a commutator, then proves that the extension is trivial for global elements.  Romo uses a simple version of this argument to prove reciprocity for a curve in \cite{FPR3}, see the next section for more details. In \cite{OZ}, Osipov and Zhu use a similar but much more general method to get higher tame reciprocity laws. They define the category of central extensions of a group by  a Picard groupoid and then get a "commutator map" with three entries for a 2-category. They then apply this specifically to semi-global adelic complexes on an algebraic surface and a certain central extension of the general linear group of a two-dimensional local field. The resulting commutator is the higher tame symbol. They then show the extension can be trivialised, and the reciprocity law follows. Note that many of these papers refer to the higher tame symbol as the Parshin symbol.

\section{Reciprocity along a Curve}

This section will discuss reciprocity laws along curves on an algebraic surface for both the Witt symbol and the higher tame symbol, again treating them separately before gluing them together as a map to the abelian Galois group. We first prove the reciprocity law for the Witt symbol, which similarly to above will take the form of an argument for the one-dimensional case, then a simple induction. For the case of a curve, the proof is much simpler than in the characteristic zero case in \cite{Mor3}, as the mixed characteristic higher local fields complicate matters greatly.

\begin{thm}
Let $f$, $g \in F$, $h \in W(F)$ and fix an irreducible curve $y \subset X$. Then for all but finitely many $x \in y$, the Witt symbol $(f,g|h]_{x,y}$ is zero, and the sum
\[ \sum_{x \in y}(f,g|h]_{x,y} = 0.\]\end{thm}

\begin{proof}
The first statement follows similarly to lemma \ref{wittzeros}. For the second statement, first let $\omega \in \Omega^{2, cts}_{F/\mathbb{F}_{q}}$. Fix a local parameter $t_{z} \in \mathcal{O}_{z}$ for $z$ a branch of $y$ and expand 
\[\omega = \sum_{i} \eta_{i}\wedge t_{z}^{i}dt_{z}\]
where $\eta_{i} \in \Omega^{1}_{k(z)/\mathbb{F}_{q}}$. Then we know by definition that $\text{res}_{x,z}(\omega) = \text{res}_{x}(\eta_{-1})$. So we have reduced to the case for a curve,
\[\sum_{x \in z}\text{res}_{x,z}(\omega) = \sum_{x \in z}\text{res}_{x}(\eta_{-1}) = 0\] which is well-known, see e.g. \cite{AG}, III.7.14. for a general Witt vector, we can now sum over all branches of $y$ and induct as in \ref{reciprocitywitt}.

\end{proof}

We treat the higher tame symbol similarly to the previous section.
\begin{thm}
Let $f$, $g$, $h \in F^{\times}$ and fix an irreducible curve $y \subset X$. Then for all but finitely many $ x \in y$ the higher tame symbol $(f,g,h)_{x,y}$ is equal to one, and the product
\[\prod_{x \in y}N_{k_{z}(x)/\mathbb{F}_{q}}(f,g,h)_{x,y} = 1.\]\end{thm}

We provide a proof similar to the proof above for the Witt symbol, showing we can reduce to the one-dimensional case. This method is well-known, but not with the norm from $k_{z}(x)$ to $\mathbb{F}_{q}$ taken into account. We also provide a discussion of other methods of proving this theorem.

\begin{proof}
By rearranging the definition of the higher tame symbol given at the start of section 3, for any $x \in z \in y(x)$ we can write
\[(f,g,h)_{x,z} = (p(t_{z}^{-v_{z}}(f)f),p(t_{z}^{-v_{z}}(g)g))_{x}^{\bar{v}_{x}(h)} \times\]
\[ (p(t_{z}^{-v_{z}}(g)g),p(t_{z}^{-v_{z}}(h)h))_{x}^{-\bar{v}_{x}(f)}
\times\] 
\[(p(t_{z}^{-v_{z}}(h)h),p(t_{z}^{-v_{z}}(f)f))_{x}^{\bar{v}_{x}(g)} \times
(-1)^{\beta_{x}}\]
where $p$, $v_{z}$ and $\bar{v}_{x}$ are all as defined in section 3, the symbol $( \ , \ )_{x}$ is the tame symbol and
\[\beta_{x} = v_{z}(f)v_{z}(g)\bar{v}_{x}(h)+v_{z}(f)v_{z}(h)\bar{v}_{x}(g) + v_{z}(g)v_{z}(h)\bar{v}_{x}(f).
\]
Now by Weil reciprocity, the product of each tame symbol $\prod_{x \in z}( \ , \ )_{x} = 1$, so we just need to check that the product $\prod_{x \in z}(-1)^{\beta_{x}}$ is also equal to one.

We have
\[\prod_{x \in z}(-1)^{v_{z}(f)v_{z}(g)\bar{v}_{x}(h)} = (-1)^{v_{z}(f)v_{z}(g)\left( \sum_{x \in z}\bar{v}_{x}(h)\right)} = 1\]
by the theory of valuations on a curve. By symmetry this works for all terms of $\beta$, and so the formula
\[\prod_{x \in z}N_{k_{z}(x)/\mathbb{F}_{q}}(f,g,h)_{x,z} = 1\] holds. Now we write
\[\prod_{x \in y} (f,g,h)_{x,y} = \prod_{x \in y}\prod_{z \in y(x)} N_{k_{z}(x)/\mathbb{F}_{q}}(f,g,h)_{x,z} = 1\]
where we may rearrange the product as needed by the argument in lemma 6.3. This completes the proof.

\end{proof}

The methods of Osipov and Zhu in \cite{OZ} also work for reciprocity along a curve. In \cite{FPR3}, Romo uses an algebraic construction to show the product of higher tame symbols around a point is one. He shows that the tame symbol is the unique continuous Steinberg symbol in the cohomology class of the commutator symbol given by the central extension
\[0 \to k(v)^{\times} \to \widetilde{\Sigma_{v}^{\times}} \to \Sigma_{v}^{\times} \to 0\]
where $\Sigma_{v}$ is the fraction field of a discrete valuation ring with residue field $k(v)$. The Parshin symbol can then be realised (up to a sign) as the composition of two of these commutators. The reciprocity law follows very easily, via an analogy to Tate's residue theorem (see \cite{T3}). Romo's approach has much in common with Osipov and Zhu, but is much more explicit about this particular case when compared to their categorical constructions.

\section{Further Work}
To extend Parshin's higher local class field theory to algebraic surfaces, it is necessary to prove some duality theorems - i.e. that the pairings of the topological $K$-groups with the Witt vectors and multiplicative group are non-degenerate on certain subgroups of ${\prod}'K_{2}^{top}(F_{x,y})$, where the restricted product is the geometric two-dimensional adeles defined in \cite{F3}. For the Witt pairing, this idea was originally developed by Kawada and Satake in \cite{KS} for a global field of positive characteristic. For the higher tame pairing, the classical case is Kummer theory.  It is expected that these pairings, along with the Frobenius for unramified extensions will give a reciprocity map for global and semi-global fields associated to an algebraic surface. \\
As in the one-dimensional case, the kernel will be the diagonally embedded global elements - but in this case the semi-global elements of $F_{x}$ and $F_{y}$ are also included. The reciprocity laws proven in this paper can be used to show that these elements are contained in the kernel of the reciprocity map, an important first step to obtaining this reciprocity map for an algebraic surface. We use the methods described in section five to define the explicit reciprocity map as the gluing together of the Witt symbol and the higher tame symbol in the absolute abelian Galois group. Once the duality theorems are proven this will then give the full details of the exact sequences of class field theory.

\bibliographystyle{plain}
\bibliography{References}
\end{document}